\documentclass[11pt,a4paper]{article}

\usepackage{mathtools}
\usepackage{amsthm}
\usepackage[a4paper, left=2.5cm, right=2.5cm, top=2cm]{geometry}
\usepackage{setspace}

\DeclarePairedDelimiter{\ceil}  {\lceil } {\rceil }
\DeclarePairedDelimiter{\floor} {\lfloor} {\rfloor}
\DeclarePairedDelimiter{\fract} {\{     } {\}     }
\DeclarePairedDelimiter{\abs}   {\lvert } {\rvert }

\newcommand{\rnk}[2]  {\floor{ \sqrt[#2] {#1}}}
\newcommand{\lnl}[2]  {\floor{ \log_{#1} {#2}}}
\newcommand{\fknl}[3] {f_{#1} \left( #2, #3 \right) }
\newcommand{\Fkxl}[3] {F_{#1} \left( #2, #3 \right) }

\newcommand{\sto}[2]  {\genfrac{[} {]} {0pt} {} {#1} {#2} }

\newtheorem{myTheorem}{Theorem}[section]
\newtheorem{myLemma}{Lemma}[section]
\newtheorem{myCorollary}{Corollary}[section]

\numberwithin{equation}{section}

\title{Ordered Factorizations with $k$ Factors}
\author{Jacob Sprittulla \\ sprittulla@alice-dsl.de}
\date{\today}

\begin{document}
	
\maketitle
	
\begin{abstract}
	We give an overview of combinatoric properties of the number of ordered $k$-factorizations $f_k(n,l)$ of an integer, where every factor is greater or equal to $l$. We show that for a large number $k$ of factors, the value of the cumulative sum $F_k(x,l)=\sum\nolimits_{n\leq x} f_k(n,l)$ is a polynomial in $\lnl{l}{x}$ and give explicit expressions for the degree and the coefficients of this polynomial. An average order of the number of ordered factorizations for a fixed number $k$ of factors greater or equal to 2 is derived from known results of the divisor problem. 
\end{abstract}
	
\section{Introduction}
	We study the number of ordered factorizations 
	$f_k(n,l) = \# \{   (i_1,\dots,i_k) \geq l, i_1 \cdots i_k = n\}$ 
	of a positive integer $n$ with exactly $k$ factors greater or equal to $l$, where factorizations with the same factors in different orders are considered to be different. Here $\# \{ \cdots \}$ denotes the cardinality of a set. For example, for $n=12$, $l=2$ and $k=1,2,3$ we have 
	\begin{align*}
		f_1(12,2) & =1 = \# \fract{(12)}                    \\
		f_2(12,2) & =4 = \# \fract{(2,6),(6,2),(3,4),(4,3)} \\
		f_3(12,2) & =3 = \# \fract{(2,2,3),(2,3,2),(2,2,3)} \text{.}
	\end{align*}
	
	We are mainly interested in the cases $l=1$ and $l=2$, but some properties rely on the recursive structure of the functions $f_k(n,l)$ for $l > 2$ (see theorem~\ref{th:combi} below), so that it is useful to treat the minimal admissible value $l$ for the factors as a separate parameter. In some studies, cf. \cite{Hwa11} or \cite{War93} for example, the set of admissible factors is further constrained, but we restrict ourselves to the case of factors greater or equal to a minimal value $l$. 
	
	To simplify the notation, we omit the parameter $l$ for $l=1$ and $l=2$ and use the notations $d_k(n):=f_k(n,1)$ and $f_k(n):=f_k(n,2)$. We denote the corresponding summatory functions with capital letters and write
	$F_k(x,l) := \sum\nolimits_{n\leq x} f_k(n,l)$, 
	$F_k(x)   := \sum\nolimits_{n\leq x} f_k(n)$ and 
	$D_k(x)   := \sum\nolimits_{n\leq x} d_k(n)$ 
	for real $x \geq 1$.
	
	Properties of ordered factorizations have a long history in the mathematical literature. We refer to \cite{Kno05} and \cite[section 4]{Kla07} for good overviews. 
	
	An explicit formula for $f_k(n)$ was given by MacMahon in \cite{Mac93}, compare also \cite{Kuh50}. If the prime factorization of an integer $n$ is given by 
	$n=p_1^{e_1} p_2^{e_2} \cdots p_{\omega(n)}^{e_{\omega(n)}}$, 
	where $\omega(n)$ denotes the number of distinct prime factors of $n$, MacMahon's explicit formula is given by 
	\begin{equation} 
	\label{eq:Mac}
		f_k(n) = \sum_{i=0}^{k-1} (-1)^i \binom{k}{i} 
		         \prod_{j=1}^{\omega(n)} \binom{e_j+k-i-1}{e_j} 
		         \text{.}
	\end{equation}
	This formula in combination with \eqref{eq:dnk} below can also be used to calculate $d_k(n)$ explicitly. 
		
	Most of the studies of ordered factorizations focus on the cumulative function 
	$f(n):=\sum\nolimits_{k=1}^{\infty} f_k(n)$
	counting all ordered factorizations, also called the Kalmar function. Kalmar in~\cite{Kal32} proved an asymptotic of the form
	\begin{align}
		F(x):= \sum_{n \leq x} f(n)
		     = K x^\rho + \bigtriangleup(x)  \text {,}
		\label{eq:kal} 
	\end{align}
	where the parameters of the main term are given by 
	$\rho = \zeta^{-1}(2) \approx 1.7286$ and 
	$K    = - \left( \rho \zeta'(\rho) \right)^{-1} \approx 0.31817$ and 
	$\zeta(\cdot)$ denotes the Riemann zeta function. The order of the error term in \eqref{eq:kal} has been improved in several steps, the currently best known result is given in~\cite{Hwa00}. 
	
	Lower and upper bounds for $f(n)$ are studied in \cite{Cho00}, \cite{Cop05} and \cite{Kla07}. In \cite{Del08} results are given for $f$-champions, i.e. integers $N$ for which $f(N) > f(n)$ for all $n < N$.
	
	The functions $f_k(n)$ resp. $F_k(x)$ are explicitly treated in \cite{Hwa00}, \cite{Hwa11} and \cite{Lau01}. In \cite{Hwa11} a central limit theorem for $F_k(x,l)$ for $x \rightarrow \infty$ is proven\footnote{In fact, the result proven in \cite{Hwa11} is more general, since it covers factorizations with  constraints.}. Results on the average order of $f_k(n)$ for $k \geq 2$ are given in \cite{Hwa00} and \cite{Lau01}. We come back to these results in section~\ref{sec:avgorder} below. 

	It is worth mentioning that the functions $f_k(n)$ and $F_k(x)$ are directly connected to some of the most important arithmetical functions. We denote by $\mu(n)$ the Moebius function, by $M(x)=\sum_{n \leq x} \mu(n)$ the Mertens function, by $\Lambda(n)$ the van Mangoldt function and by 
	$\Pi(x)=\sum_{n \leq x} \frac{\Lambda(n)}{\log n}$ the Riemann prime counting function. We have for $n,x \geq 1$ (see \cite[chapter 17.2]{Fri10})
	\begin{align}
		\mu(n) &= \sum_{k=0}^{\log_2 n} (-1)^k f_k(n)   
		\label{eq:fmoe}                                    \\
		 M(x)  &= \sum_{k=0}^{\lnl{2}{x}} (-1)^k F_k(x)   
		\label{eq:fmer} 									\\
		\frac{\Lambda(n)}{\log n} &= \sum_{k=1}^{\log_2 n} 
		         \tfrac{1}{k} (-1)^{k+1} f_k(n)   \\
		\Pi(x)   &= \sum_{k=1}^{\lnl{2}{x}} 
	              \tfrac{1}{k} (-1)^{k+1} F_k(x)  \text{,} 
	\end{align}
	where the conventions of \eqref{eq:fnkl0} and \eqref{eq:Fnkl0} below for values at $k=0$ are used. From equation \eqref{eq:fmer} it follows that the Mertens function at $x$ can be regarded as the surplus of the number of factorizations of integers smaller or equal to $x$ with an even number of factors over the number of factorizations with an odd number of factors.
	
	The aim of this paper is threefold. First, we want to give a systematic overview of the recursive structure of the quantities $F_k(x,l)$ and $f_k(n,l)$. We do not claim that any of the given formulas is new, but a complete overview does not seem to exist in the literature. Recursive formulas are covered in section \ref{sec:formulas}.

	In section \ref{sec:polynomial} we exploit the recursive structure of $F_k(x,l)$ to derive explicit polynomial type formulas when the number of factors $k$ is near its maximum value $\lnl{l}{x}$, for $l \geq 2$. Our results generalize an observation in \cite[section 8]{Hwa00}. 
	
	In section \ref{sec:avgorder} we consider the average order of $f_k(n)$ for fixed $k$. Although the results given here are straightforward implications of well known asymptotics of the divisor problem and the fact that $D_k(x)$ is the binomial transform of $F_k(x)$ (see \eqref{eq:Dnk} below), it seems that the resulting average orders for $f_k(n)$ haven't yet been discussed in the literature. 
	
	\textbf{Notations:} $i,j,k,l,n,m$ always denote positive integers, $x, y, u, v, w$ real numbers and $s, z$ complex numbers. We write $\sigma_s$ for the real part of $s$. As usual, $\floor{x}$ denotes the floor function (the greatest integer smaller than $x$), $\ceil{x}$ denotes the ceiling function (the smallest integer greater than $x$) and $\fract{x}=x-\floor{x}$ denotes the fractional part of $x$. The Riemann zeta function is denoted by $\zeta(s)$. We also use the notation $\zeta_l(s)=\sum_{n=l}^{\infty} n^{-s}$  ($\sigma_s > 1$) for the truncated Riemann zeta function. Empty sums are considered to be zero.
	
\section{Combinatoric identities for $f_k(n,l)$ and $F_k(x,l)$}
\label{sec:formulas}
	We first note that, since $l^k > n$ for $k > \lnl{l}{n}$ or $l > \rnk{n}{k}$, we have for $k,l \geq 2$
	\begin{align}
	\label{eq:f0}
		f_k(n,l) = F_k(n,l) = 0 \quad 
		\text{for} \quad k > \floor{\log_l n} \quad
		\text{or}  \quad l > \rnk{n}{k}
		\text{.}
	\end{align}
	We also have $f_k(n,l)=0$ for $k > \Omega(n)$, where $\Omega(n) \leq \lnl{2}{n}$ denotes the total number of prime factors of $n$.
	
	For $k=1$ we have
	\begin{align}
	\label{eq:f1}
		f_1(n,l) = 
		\begin{cases}
			0 & \text{for} \quad  n <     l \\
			1 & \text{for} \quad  n \geq  l
		\end{cases}	
		\text{,}  \qquad
		F_1(x,l) = \left( \floor{x} - l + 1 \right)^+
		\text{,}
	\end{align}
	where $y^+ := \max(0,y)$. 
	From the definition it is clear, that for $n,x \geq 1$ and $k,l \geq 1$ 
	\begin{align}
		f_k(n,l) &= \sum_{\substack{i=l \\ i \lvert n}}^n f_{k-1}(n/i,l) 
		\text{,} \qquad  
		f_0(n,l) =
		\begin{cases}
			1 & \text{for} \quad  n=1 \\
			0 & \text{for} \quad n \geq  2
		\end{cases}
		\label{eq:fnkl0}    \\  
		F_k(x,l) &= \sum_{i=l}^{n} F_{k-1}(x/i,l)
		\text{,} \qquad  F_0(x,l) = 1
		\text{.}
		\label{eq:Fnkl0}    
	\end{align}
	
	For concrete calculations, these recursive expressions are of limited use due to their computational extensiveness. Note that \eqref{eq:fnkl0} can be written as $f_k(n,l) = f_{k-1}(n,l) \ast f_{1}(n,l)$, where~$\ast$ denotes Dirichlet convolution. If we denote by $\mathcal{F}_{k,l}(s)$ the Dirichlet generating function of $f_k(n,l)$, it follows that 
	$\mathcal{F}_{k,l}(s) = \mathcal{F}_{k-1,l}(s) \zeta_l(s)$
	and therefore, for $k,l \geq 1$ and $\sigma_s>1$ (compare, for example~\cite{Hwa11})
	\begin{align}
		\mathcal{F}_{k,l}(s) &= \sum_{n=1}^\infty f_k(n,l) n^{-s}  
				   = \zeta_l(s)^k
				   \text{.}
		\label{eq:dirf}
	\end{align}
	
	By uniqueness of the coefficients of the Dirichlet series, equation \eqref{eq:dirf} can serve as a definition of $f_k(n,l)$ (see \cite{Hwa00}, for example).
	
	In some circumstances it might be useful to use the hyperbola method (cf. \cite[Theorem~I.3.1]{Ten95}) for concrete calculation of $F_k(x,l)$. For $uv=x$, $l \geq 2$ and $0 \leq j \leq k$, we use $f_k(n,l) = f_{k-j}(n,l) \ast f_{j}(n,l)$ and \eqref{eq:Fnkl0} to get
	\begin{equation*} 
	F_k(x,l) = \sum_{i=1}^{u} F_{k-j}(x/i,l) f_j(i,l) +
			   \sum_{i=1}^{v} F_{j}  (x/i,l) f_{k-j}(i,l) - 
	           F_j(u,l) F_{k-j}(v,l) 
			   \text{.}
	\end{equation*}
	
	This allows, for example, an efficient calculation of $F_{2k}(n,l)$ if $F_{k}(i,l)$ for $i=1,\dots,n$ is already known:
	\begin{equation*} 
		F_{2k}(n,l) = 2 \sum_{i=1}^{\floor{\sqrt[k]{n}}} 
		F_k(n/i,l) f_{k}(i,l) 
		- F_k \left( \floor{\sqrt[k]{n}},l \right) ^2 \text{.}
	\end{equation*}
	
	Another useful special case is the relation 
	$F_2(n,l)= 2 \sum_{i=1}^{\floor{\sqrt{n}}} \floor{n/i} 
	           - \floor{\sqrt{n}}^2+(l-1)^2$.
	
	The following theorem covers the recursive structure of the functions $f_k(n,l)$ and $F_k(x,l)$.
	\begin{myTheorem} 
	\label{th:combi}
		For $x,n \geq 1$ and $k,l \geq 1$ we have
		\begin{align} 
			f_k(n,l) & =\sum_{\substack{i=0 \\ l^i \lvert n}}^{k} \binom{k}{i}  \fknl{k-i}{\frac{n}{l^i}}{l+1} 
			\label{eq:fnkl1}  \\
			F_k(x,l) & =\sum_{i=0}^{k} \binom{k}{i} 
			\Fkxl{k-i}{\frac{x}{l^i}}{l+1}
			\text{.}
			\label{eq:Fnkl1}  
		\end{align}
		
		Further, for $x,n,l$ as above and $k \geq 2$ we have
		\begin{align} 
			f_k(n,l) & = \sum_{m=l}^{\rnk{n}{k}} 
			\sum_{\substack{i=1 \\ m^i \lvert n}}^{k} \binom{k}{i} 
			\fknl{k-i}{\frac{n}{m^i}}{m+1} 
			\label{eq:fnkl2}  \\
			F_k(x,l) & =\sum_{m=l}^{\rnk{x}{k}} \sum_{i=1}^{k} \binom{k}{i} \Fkxl{k-i}{\frac{x}{m^i}}{m+1} 
			\text{.}
			\label{eq:Fnkl2}   
		\end{align}	
	\end{myTheorem}
	\begin{proof}
		We first give a combinatoric proof of \eqref{eq:fnkl1} and \eqref{eq:Fnkl1}. The basic idea is the separation of factors equal to $l$. For fixed $n,k,l$ and $0 \leq i \leq k$, we denote by $f_{k,i}(n,l)$ the number of factorizations of $n$, where all $k$ factors are greater or equal to $l$ and exactly $i$ factors are equal to $l$. If $l^i$ divides $n$, we have 
		\begin{equation*}
			f_{k,i}(n,l) = 
			\binom{k}{i}  f_{k-i} \left( \frac{n}{l^i},l+1 \right) \text{,}
		\end{equation*}	
		because every factorization counted by $f_{k,i}(n,l)$ can be split into $i$ factors equal to $l$ and $k-i$ factors greater or equal to $l+1$. 
		
		A simular argument gives
		\begin{equation*}
			F_{k,i}(x,l) = 
			\binom{k}{i}  F_{k-i} \left( \frac{x}{l^i},l+1 \right) 
		\end{equation*}	
		for $0 \leq i \leq k$, where $F_{k,i}(x,l):=\sum_{n \leq x} f_{k,i}(n,l)$ counts all factorizations of integers less or equal to $x$, with $k$ factors, where $i$ factors are equal to $l$ and $k-i$ are greater or equal to $l+1$. Finally we get \eqref{eq:fnkl1} and \eqref{eq:Fnkl1} from 		
		$f_k(n,l) = \sum_{i=0}^{k} f_{k,i}(n,l)$ and
		$F_k(x,l) = \sum_{i=0}^{k} F_{k,i}(x,l)$. 
		 
		We proceed to show \eqref{eq:Fnkl2}, by subsequent elimination of the first term of the right hand side of \eqref{eq:Fnkl1}. More precisely, we separate the first term in the sum of \eqref{eq:Fnkl1} and apply \eqref{eq:Fnkl1} again (with $l+1$ as second argument of $F_k(\cdot)$) to this term to get  
		\begin{align*}
			F_k(x,l) &= F_{k}\left( x,l+1 \right) + \sum_{i=1}^{k} \binom{k}{i} F_{k-i}\left( \frac{x}{l^i},l+1 \right) \\
			&= 	\sum_{i=0}^{k} \binom{k}{i} F_{k-i}\left( \frac{x}{(l+1)^i},l+2 \right) + \sum_{i=1}^{k} \binom{k}{i} F_{k-i}\left( \frac{x}{l^i},l+1 \right)  \\
			&= F_{k}\left( x,l+2 \right) + \sum_{m=l}^{l+1} \sum_{i=1}^{k} \binom{k}{i} F_{k-i}\left( \frac{x}{m^i},m+1 \right) 
			\text{.}
		\end{align*}
		Repeating the above operation $j$-times yields
		\begin{align*}
			F_k(x,l) 	&= F_{k}\left( x,l+j+1 \right) + \sum_{m=l}^{l+j} \sum_{i=1}^{k} \binom{k}{i} F_{k-i}\left( \frac{x}{m^i},m+1 \right) 
			\text{.}
		\end{align*}
		Setting $j = \rnk{n}{k} - l$ and using \eqref{eq:f0}	we get \eqref{eq:Fnkl2}. \\
		An analogous argument yields \eqref{eq:fnkl2}.
		This completes the proof.
	\end{proof} 
	For practical purposes, the performance of the recursions of theorem \ref{th:combi} is in most parameter constellations much better than the performance of the recursions \eqref{eq:fnkl0} and \eqref{eq:Fnkl0}. However, for large values of $n,x$ the recursions tend to be numerically unstable.
	
	The case $l=1$ of theorem~\ref{th:combi} connects $F_k(x)$ and $D_k(x)$, respectively $f_k(n)$ and $d_k(n)$ .
	\begin{myCorollary} \label{co:combi1}
		For $x,n \geq 1$ and $k\geq 0$ we have
		\begin{align}
		D_k(x) &= \sum_{i=0}^{\floor{\log_2 x }}      \binom{k}{i} F_i(x)        \label{eq:Dnk} \\
		F_k(x) &= \sum_{i=0}^{k}      (-1)^{k-i} \binom{k}{i} D_i(x) 
		\label{eq:Fnk} \\
		d_k(n) &= \sum_{i=0}^{\floor{\log_2 n }}      \binom{k}{i} f_i(n)        \label{eq:dnk} \\
		f_k(n) &= \sum_{i=0}^{k}      (-1)^{k-i} \binom{k}{i} d_i(n)  
		\label{eq:fnk} 
		\text{.}
		\end{align}
	\end{myCorollary}
	\begin{proof}
		The relations \eqref{eq:Dnk} and \eqref{eq:dnk} follow directly from \eqref{eq:Fnkl1} and \eqref{eq:fnkl1} with $l=1$ and the boundary conditions \eqref{eq:f0}. 
		
		By the definition of the binomial transform, we can say that for fixed $x \geq 1$ (resp. $n \geq 1$), $D_k(x)$ (resp. $f_k(n)$) is the binomial transform (with the respect to $k$) of $F_k(x)$ (resp. $f_k(n)$). Therefore the relations \eqref{eq:Fnk} and \eqref{eq:fnk} can be deduced from the inversion of the binomial transform in general. 
	\end{proof} 
	\textbf{Remark 1:} The relationship between $f_k(n)$ and $d_k(n)$ covered by corollary \ref{th:combi} seems to be well known, for example equation \eqref{eq:fnk} is mentioned in \cite[Chapter 17.2]{Fri10}. Equation \eqref{eq:dnk} appears in a footnote of \cite{Skl52}.
	
	\textbf{Remark 2:} In this paper, we restrict ourselves to the case of factorizations where all integers greater or equal to a given $l$ are allowed, since we are mainly interested in the case $l=1$ and $l=2$. The above formulas in theorem \ref{th:combi} and corollary \ref{co:combi1} could be generalized to the case of factorizations consisting of arbitrary subsets of the positive integers (with at least two elements), as treated in \cite{Hwa11} or \cite{War93}. The main idea in the proof of theorem  \ref{th:combi} is to separate the smallest factor in the factorizations, which is also possible in the general (constrained) case. Similar results as in corollary \ref{th:combi} hold whenever $1$ is the (smallest) element of the set of admissible factors.
	
	Another remarkable relation between $f_k(n)$ and $d_k(n)$ is treated in the next corollary.
	\begin{myCorollary} 
	\label{co:combi2}
		For $n \geq 1$ and $\abs{u} > 1$ we have
		\begin{align}
		\sum_{k=0}^{\infty} u^{-k} d_k(n) &= 
		\frac{u}{u-1} \sum_{k=0}^{\floor{\log_2 n}} (u-1)^{-k} f_k(n) 
		\label{eq:dfnk}  
		\end{align}
	\end{myCorollary}
	\begin{proof}
		
		Recall that for given $k \geq 1$ the generating function of the binomial coefficients is given by
		\begin{align}
			\sum_{i=k}^{\infty} \binom{i}{k} y^i 
			= \frac{y^k}{(1-y)^{k+1}} \text{,}
			\label{eq:gfbin}
		\end{align}
		with absolute convergence for $\abs{y} < 1$. 
		
		For $n \geq 1$, $\abs{y} < 1$ and large $N$, we have by \eqref{eq:dnk} and lemma \ref{lm:bin} (with $r=0$) below 
		\begin{align}
			\sum_{k=0}^{N} y^{k} d_k(n)
			 &= \sum_{k=0}^{N} y^{k} \sum_{i=0}^{k} \binom{k}{i} f_i(n) 
			 \notag         \\
			 &= \sum_{k=0}^{N} f_k(n) \sum_{i=k}^{N} \binom{i}{k} y^{i}
			 \text{.}
			 \label{eq:fins}
		\end{align}
		Using \eqref{eq:gfbin}, by absolute convergence we can let $N \rightarrow \infty$ in \eqref{eq:fins} to get
		\begin{align*}
			\sum_{k=0}^{\infty} y^{k} d_k(n)
			&= \sum_{k=0}^{\infty} \frac{y^k}{(1-y)^{k+1}} f_k(n) 
		\end{align*}
		Finally, we set $u:=\tfrac{1}{y}$ and the claim follows by factoring out 
		$\tfrac{1}{1-y}=\tfrac{u}{u-1}$ and taking into account $\tfrac{y}{1-y} = (u-1)^{-1}$.
	\end{proof}
	Note that in \eqref{eq:dfnk} $d_k(n)$ and $f_k(n)$ can be replaced by $D_k(x)$ and $F_k(x)$, for $x \geq 1$, by the definition of $F_k(x,l)$ as the cumulated sum over $f_k(n,l)$, $n \leq x$.
	
	Special cases of \eqref{eq:dfnk} include the equation 
	$2 f(n)=\sum_{k=0}^{\infty} 2^{-k} d_k(n)$  
	for $u=2$. This formula was proved by Sen in \cite{Sen41} for the special case of square free $n$ and then later used by Sklar in \cite{Skl52} to derive an asymptotic for $f(n)$ in this case.

\section{Factorizations with a large number of factors}
\label{sec:polynomial}
	Throughout this section we use the notation $t=t(x,l)=\floor{\log_l{x}}$ for given $x$ and $l$. In this section \eqref{eq:Fnkl1} will be applied to show that $F_{t-j}(x,l)$ is a polynomial in $t$; we give explicit formulas for the degree $\tau$ and the coefficients of the polynomial. 

	We begin by preparing two lemmas. The first lemma exploits the fact that $F_k(n,l)$ vanishes for large $k$ and gives an explicit expression for the number of summands in \eqref{eq:Fnkl1}. 
	
	\begin{myLemma}
	\label{lm:ftau}
		For $x\geq 1$, $l \geq 1$, $k\geq 1$ we have 
		\begin{align}
			F_k(x,l) &=\sum_{i=0}^{\tau(x,k,l)} 
			\binom{k}{i}F_{i} \left( \frac{x}{l^{k-i}},l+1 \right) 
			\text{,} \quad \text{with} 
			\label{eq:tauf} \\
			\tau(x,k,l) & =\min \left( k, 
			\left \lceil{ \frac{\log x - k \log l}{\log(l+1) - \log l}}
			\right \rceil  \right)
			\text{.}
			\label{eq:tau}
		\end{align} 	
	\end{myLemma}
	\begin{proof}
		First note that in \eqref{eq:tauf} we have reversed the order of summation in comparison to \eqref{eq:Fnkl1} and used the fact that $\binom{k}{i} =\binom{k}{k-i}$. From \eqref{eq:f0}, the term $F_{i}(\frac{x}{l^{k-i}},l+1)$ vanishes if either 
		\begin{align*}
			l+1 > \sqrt[i]{n/l^{k-i}} \text{}  \qquad \text{or} \qquad
			  i > \log_{l+1} (n/l^{k-i})
			  \text{.}
		\end{align*}
		After some algebra, this leads in both cases to 
		\begin{align*}
			i &> \frac{\log n - k \log(l+1)}{\log l - \log(l+1)}
			\text{.}
		\end{align*}
		This completes the proof.
	\end{proof} 

	The next lemma was already used in the proof of corollary \ref{co:combi2}.
	\begin{myLemma}
		\label{lm:bin}
		For real $\nu_i$, $\xi_{i,j}$ and $\gamma_j$, we have for $k \geq 1$ and 
		$0 \leq r \leq k$ 
		\begin{align*}
		\sum_{i=r}^{k} \nu_i \sum_{j=0}^{i-r} \xi_{i,j} \gamma_j
		&= \sum_{j=0}^{k-r} \gamma_j  
		\sum_{i=j+r}^{k} \nu_{i} \xi_{i,j}
		\text{.} 
		\end{align*}
	\end{myLemma}
	
	\begin{proof}
		We write out the left hand side of the equation and rearrange terms to get
		\begin{align*}
		\sum_{i=r}^{k} \nu_i \sum_{j=0}^{i-r} \xi_{i,j} \gamma_j
		&= \nu_r     (\xi_{r,0}   \gamma_0 ) 
		+ \nu_{r+1} (\xi_{r+1,0} \gamma_0 + \xi_{r+1,1} \gamma_1 ) 
		+ \cdots 
		+ \nu_k (\xi_{k,0} \gamma_0 + \dots + \xi_{k,k-r} \gamma_{k-r}) \\
		&= \gamma_0     ( \nu_r     \xi_{r,  0} + \cdots + \nu_k \xi_{k,0} ) + 
		\gamma_1     ( \nu_{r+1} \xi_{r+1,1} + \cdots + \nu_k \xi_{k,1} ) + \cdots +
		\gamma_{k-r} ( \nu_k \xi_{k,k-r} ) 
		\text{}
		\end{align*}
		and the claim follows.
	\end{proof}	
	The next theorem is a straightforward implication of \eqref{eq:Fnkl1} and the fact that for positive integers $n \geq k$, we have $\binom{n}{k} =  \sum_{i=0}^{k} n^i \sto{k}{i}  \frac{1}{k!}$, where $\sto{k}{i}$ denotes the Stirling numbers of the first kind.
	\begin{myTheorem}
	\label{th:fpoly}
		For $x \geq 1$, $l \geq 1$, $t=\floor{\log_l x}$ and $j \leq t-1$, we have $F_{t-j}(x,l)=P_{\tau}(t-j)$, where $P_{\tau}$ is a polynomial of degree $\tau$ in $t-j$ given by 	
		\begin{align*}
		P_\tau(t-j)&=\sum_{m=0}^{\tau} v_m (t-j)^m 
		\text{,} \quad \text{with} \\
		\tau(x,j,l) & =\min \left( t-j, 
		\left \lceil{ \frac{\fract{log_l x} + j \log l}{\log(l+1) - \log l}}
		\right \rceil  \right) \\
		v_m &= \sum_{i=m}^{\tau} \kappa_i \sto{i}{m} \frac{1}{i!}    \\
		\kappa_{i} &= F_i \left( x_i, l+1 \right) \\
		x_i &= l^{\fract{log_l x}} l^{i+j}\text{,} \quad i=0,\dots,\tau
		\text{.}
		\end{align*}
	\end{myTheorem}
	\begin{proof}
		It follows from \eqref{lm:ftau} that 
		\begin{align*}
			F_{t-j}(x,l) &=\sum_{i=0}^{\tau(x,j,l)} 
			\binom{t-j}{i} \Fkxl{i}{\frac{x}{l^{t-j-i}}}{l+1} 
			\text{,}  \quad \text{with} \\ 
			\tau(x,j,l)  &=\min \left( t-j, 
			\left \lceil{ \frac{\log x - (t-j) \log l}{\log(l+1) - \log l}}
			\right \rceil  \right)
			 \text{.}
		\end{align*}
		The expression for $\tau(x,j,l)$ in the theorem follows from 
		\begin{align*}
			\log x - (t-j) \log l = \log x - (\floor{\log_l x}-j) \log l  
			= \fract{\log x} +  j \log l 
			\text{.}
		\end{align*}
		The first argument of $F_i$ in the sum becomes
		\begin{align*}
		\frac{x}{l^{t-j-i}} &= \frac{x}{l^{\floor{log_l x}-j-i}}  
		= l^{log_l x} l^{- \floor{log_l x}} l^{i+j} 
		= l^{\fract{log_l x}} l^{i+j} 
		= x_i
		\text{,}
		\end{align*}
		so that we get $F_{t-j}(x,l) =\sum_{i=0}^{\tau(x,j,l)} 
		\binom{t-j}{i} \kappa_i$, with $\tau(x,j,l)$ and 
		$\kappa_i = F_i(x_i,l+1)$ as required. Next, we use the fact that the binomial coefficients in this expression can be written as a polynomial
		\begin{align*}
		\binom{t-j}{i} &=  \sum_{m=0}^{i} (t-j)^m \sto{i}{m}  \frac{1}{i!} \text{,}
		\end{align*}
		so that we get 
		\begin{align*}
		F_{t-j}(x,l) &=\sum_{i=0}^{\tau} \kappa_i \sum_{m=0}^{i} (t-j)^m \sto{i}{m}  \frac{1}{i!}  \\
		&= \sum_{m=0}^{\tau} (t-j)^m 
		   \sum_{i=m}^{\tau} \kappa_i  \sto{i}{m}  \frac{1}{i!} \text{,}
		\end{align*}
		where we have used lemma \ref{lm:bin} (with $r=0$) in the last equation.
		This completes the proof.
	\end{proof} 
	The calculation of the coefficients $v_m$ of the polynomial in theorem \ref{th:fpoly} requires the calculation of Stirling numbers of the first kind and of values of $F_k(n,m)$ for parameters $m \geq l$, which is, in principle, possible via theorem \ref{th:combi}. An easier method to derive the coefficients is given in the next corollary.
	\begin{myCorollary}
	\label{co:fpoly}
		For given $x,l$ and j as in theorem \ref{th:fpoly}, we set $\lambda_i := F_i(x_i,l)$, with $x_i$ ($i=0,\dots,\tau$) and $\tau$ defined as in theorem \ref{th:fpoly}. 
		Let $\lambda$ be the vector of the $\lambda_i$'s and $B$ be the matrix of $b_{i,m} = (i+j)^{m}$ for $i,m=0,\dots,\tau$. Then
		\begin{align}
			F_{t-j}(x,l)=\sum_{m=0}^{\tau} w_{m} t^m \text{,}	
		\end{align} 
		where the coefficients $w_m$ are the elements of the vector $w=\lambda B^{-1}$ and $t=\lnl{l}{x}$.	
	\end{myCorollary}
	\begin{proof}
		For given $j$ and $l$, we know from theorem \ref{th:fpoly} that $\Fkxl{\lnl{l}{y}-j}{y}{l}=\sum_{m=0}^{\tau} w_{m} \lnl{l}{y}^m$ for some coefficients $w_m$ for all $y \geq 1$, where the $w_m$ depend only on the value of $\fract{\log_l y}$. Therefore, for a given $x$, we can choose $x_i = l^{\fract{log_l x} +i+j}$, for $i=0,\cdots,\tau$, with 
		\begin{align*}
			\Fkxl{i}{x_i}{l} = \sum_{m=0}^{\tau} w_{m} (i+j)^m
		\end{align*}
		since $\lnl{l}{x_i} = i+j$. Defining the vectors $\lambda, w$ and the matrix $B$ as in the corollary, the above equation reeds 
		$\lambda = wB$. Since $B$ is invertible, we finally get $w=\lambda B^{-1}$. 	\\	
		This completes the proof.
	\end{proof}

	\textbf{Example 1:} We calculate $F_{329}(10^{100})$ based on the above formulas. We have $k=329$, $t=\lnl{2}{10^{100}}=332$ and therefore $j=3$. Lemma \ref{lm:ftau} gives $\tau=5$. 
	
	Next we calculate, according to theorem \ref{th:fpoly}, $\kappa_i=\Fkxl{i}{x_i}{3}=(1,16,36,32,15,1)$, with $\floor{x_i}=(9,18,36,73,146,292)$ for $i=0,\dots,5$. With these values, we proceed to calculate the coefficients $v_i$ of the polynomial in $k$. 
	
	With corollary \ref{co:fpoly}, we calculate $\lambda_i = F_i(x_i)=(1,17,69,189,424,837)$ for $x_i$ as above and proceed to calculate the coefficients $w_i$ of the polynomial in $t$. 
	
	Finally we get the following two polynomials 
	\begin{align*}
		F_{329}(10^{100}) &= \tfrac{1}{120} k^5 + \tfrac{13}{24} k^4 +
							 \tfrac{15}{8} k^3  + \tfrac{203}{24} k^2 + \tfrac{307}{60} k + 1  \\
					  &= \tfrac{1}{120} t^5 + \tfrac{5}{12} t^4 -
						  \tfrac{31}{8} t^3  + \tfrac{223}{12} t^2 - \tfrac{252}{15} t + 53 \\
					  &=38{,}535{,}596{,}289  \text{.}
	\end{align*}
	
	\textbf{Example 2:} By calculating the polynomials at $n=2^m$ and $n=2^{m+1}-1$ for $m=0,1,\dots$ and $l=2$ with corollary \ref{co:fpoly}, we can get explicit lower and upper bounds for $F_{t-j}(n)$, using the monotonicity of $F_k(\cdot)$:
	\begin{align*}
		 1    &\leq F_{t-0}(n)  \leq t+1 \\
		2 t-1   &\leq F_{t-1}(n) \leq
		 \tfrac{1}{6} t^3 + \tfrac{3}{2} t^2 -\tfrac{2}{3} t    \\
		 \tfrac{1}{6} t^3 + \tfrac{3}{2} t^2 - \tfrac{14}{3} t + 3   & \leq F_{t-2}(n) \leq
		\tfrac{1}{120} t^5 + \tfrac{1}{24} t^4 + \tfrac{49}{24} t^3 - \tfrac{253}{24} t^2   +\tfrac{449}{20} t -19  \text{.}
	\end{align*}

\section{An average order of $f_k(n)$}
\label{sec:avgorder}
	An average order of $f_k(n)$ is given by Hwang in \cite[Corollary 3]{Hwa00} as
	\begin{equation}
		F_k(x) = x \frac{ (\log x)^{k-1}}{(k-1)!} 
				 \left( 1 + O\left( \frac{k^2}{\log x} \right) \right) \text{,}
		\label{eq:Hwa}
	\end{equation}
	with $1 \leq k = o \left( (\log x)^{2/3} \right)$. Lau \cite[Theorem 2]{Lau01} was able to improve the error to 
	$O\left( x^{1-\alpha_k} (\log x)^{k-1 }\right)$, with 
	$\alpha_k = \epsilon k^{-2/3}$, for some $\epsilon > 0$ and 
	$1 \leq k \leq \left( (\log x)^{3/5} \right)$, 
	but in his formula the main term is only specified up to some unknown constants. Note that in both approaches the parameter $k$ is allowed to grow with $x$. We treat the easier case of  finite values for $k$ here.
	
	Our approach to determine an average order of $f_k(n)$ for fixed $k$ relies on the fact that for fixed $n$, $f_k(n)$ is the (inverse) binomial transform of $d_k(n)$, see \eqref{eq:Fnk} and \eqref{eq:fnk}. For the average order of $d_k(n)$, the following theorem is known, see \cite[Chapter 13]{Ivi85}. 
	
	We use the notation $t=\log x$ for the rest of this section.
	
	\begin{myTheorem}
	\label{th:DDP}
		For $k \geq 1$, $\epsilon > 0$ and $\alpha_{k} = \frac{k-1}{k}$, there exist $a_{k,j}$, $j=0,\dots,k-1$ with 
		\begin{align}
			D_k(x) &= x P_k^D(\log x) + \bigtriangleup_k^D(x) 
			\text{, where} \\
			P_k^D(t) &= \sum_{j=0}^{k-1} a_{k,j} t^j                     \\
			\bigtriangleup_k^D(x) &= O(x^{\alpha_k + \epsilon}) 
			\text{.}
		\end{align}
	\end{myTheorem}
	
	Note that for $k=1$ we have $D_1(x) = \floor{x}$ by \eqref{eq:f1} and therefore $\bigtriangleup_1^D(x) \leq 1$, with $a_{1,0}=1$.
	
	Explicit formulas for the coefficients $a_{k,j}$ ($k \geq 2$) of the main term as functions of the Stieltjes constants are given in \cite{Lav76}. The leading terms are given by $a_{k,k-1}=\frac{1}{(k-1)!}$.
	The estimation of the error term is known as the (Dirichlet) divisor problem. The currently best known values for the exponents $\alpha_{k}$ are given in~\cite{Ivi85}. It is conjectured that $\alpha_{k} = \frac{k-1}{2k}$ holds.
	
	With this preparation, we are able to prove the following theorem for the average order of $f_k(n)$.
	\begin{myTheorem}
	\label{th:DDPF}
		For $k \geq 1$ and $\epsilon > 0$ we have $F_k(x) = x P_k^F(\log x) + \bigtriangleup_k^F(x)$ ,where
		\begin{align}
			P_k^F(t) &= \sum_{j=0}^{k-1} b_{k,j} t^j                     \\
			b_{k,j} &= \sum_{i=j+1}^{k} (-1)^{k-i} \binom{k}{i} a_{i,j}  
			\label{eq:bkj}                                               \\
			\bigtriangleup_k^F(x) &= O(x^{\beta_k + \epsilon})  
			\label{eq:trf}         \\
			\beta_k &= \max_{1 \leq j \leq k} {\alpha_j}    
			\label{eq:betakj}           
			\text{.}
		\end{align}
	\end{myTheorem}
	\begin{proof}
		First note that for $k=1$ by \eqref{eq:f1} we have 
		$F_1(x) = \left( \floor{x} -1 \right)^+$ and the claim follows.  \\
		Let $\epsilon > 0$ and $k \geq 2$ be given. From \eqref{eq:Fnk} and theorem \ref{th:DDP} we get 
		\begin{align*}
			F_k(x) &= (-1)^k + \sum_{i=1}^{k} (-1)^{k-i} \binom{k}{i} 
			         \left( x P_i^D(\log x) + \bigtriangleup_i^D(x) \right)
		\end{align*}		
		and therefore $F_k(x) = x P_k^F(\log x) + \bigtriangleup_k^F(x)$ with 
		\begin{align*}
			P_k^F(t) &= \sum_{i=1}^{k} (-1)^{k-i} \binom{k}{i}  
						\sum_{j=0}^{i-1} a_{i,j} t^j  \\
			\abs*{\bigtriangleup_k^F(x)} &\leq \abs*{ \sum_{i=0}^{k} (-1)^{k-i} \binom{k}{i}  C_{\epsilon,i}^{(D)} x^{\alpha_{i} + \epsilon}  }
		\end{align*}
		for constants $C_{\epsilon,i}^{(D)} > 0$ and $x$ large enough (the term $(-1)^k$ is asymptotically negligible). Applying lemma \ref{lm:bin} (with $r=1$) to the first equation gives formula \eqref{eq:bkj} for the coefficients of the $P_k^F$-polynomial. 
		
		Defining $\beta_{k}$ as in \eqref{eq:betakj} and $C_{\epsilon,k}^{(F)} := \sum_{i=0}^{k} \binom{k}{i} C_{\epsilon,i}^{(D)}$, we get $\abs{\bigtriangleup_k^F(x)} \leq C_{\epsilon,k}^{(F)} x^{\beta_k + \epsilon}$, which proves~\eqref{eq:trf}. 
	\end{proof}
	Note that the coefficients of the leading term in the $P_k^F$-polynomial are given by $b_{k,k-1}=\frac{1}{(k-1)!}$ and therefore the leading term coincides with the main term in \eqref{eq:Hwa}.
	
	For $k=2$ and $x \leq 2 \cdot 10^7$, we found that 
	$\abs*{\bigtriangleup_2^F(x)} < 356.1$, where the maximum value was reached at
	$x_{max}      = 19\text{,}740\text{,}240$ with 
	$F_2(x_{max}) = 334\text{,}648\text{,}770$.

\bibliographystyle{plain}

\end{document}